\documentclass[12pt]{amsart}
\usepackage{amscd, amssymb,latexsym,amsmath, amscd, amsmath}
\usepackage[all]{xy}
\usepackage{anysize}
\usepackage[sc]{mathpazo} 
\usepackage{color}
\marginsize{2.5cm}{2.5cm}{2.5cm}{2.5cm}
\newtheorem{theorem}{Theorem}[section] 
\newtheorem{lemma}[theorem]{Lemma}
\newtheorem{proposition}[theorem]{Proposition}
\theoremstyle{definition}

\newtheorem{question}{Question}
\theoremstyle{remark}
\newtheorem{remark}{Remark}

\begin{document}

\title[Simultaneous nonvanishing of the correlation coefficient]{Simultaneous nonvanishing of the correlation constant}

\author{U. K. Anandavardhanan}

\address{Department of Mathematics, Indian Institute of Technology Bombay, Mumbai - 400076, India.}
\email{anand@math.iitb.ac.in}


\date{}

\begin{abstract}
For $q=p^m$, where $p$ is an odd prime number, we study the correlation coefficient $c(\pi;H,K)$ of an irreducible (complex) representation $\pi$ of $G={\rm GL}_2(\mathbb F_q)$ with respect to a split torus $H$ and a non-split torus $K$. We consider a family of non-split tori $K_{\alpha,u}$ indexed by $u \in \mathbb F_q$ and $\alpha \in \mathbb F_q^\times \setminus \mathbb F_q^{\times 2}$. We show that under any identification of $\mathbb C$ with $\overline{\mathbb Q}_p$, and writing $\pi = \pi_r$ where $0 \leq r \leq (q-1)/2$ depending on this identification, we have 
\[c(\pi_r;H,K_{\alpha,u}) \equiv [P_r(u/\sqrt{\alpha})]^2 \mod p, \]
where $P_r(X) \in \mathbb Z[\frac{1}{2}][X]$ is the $r$-th Legendre polynomial. As a corollary, when $m \geq 2$, we prove that there exists $u \in \mathbb F_q^\times$ such that $c(\pi;H,K_{\alpha,u}) \neq 0$ for all irreducible representations $\pi$ of $G$ admitting fixed vectors for both $H$ and $K$.
\end{abstract}

\maketitle

\section{Introduction}\label{intro}

For a group $G$ with two given compact Gelfand subgroups $H$ and $K$ and an irreducible representation $\pi$ of $G$, let $v_H$ (resp. $v_K$) denote an $H$-invariant (resp. $K$-invariant) vector of unit norm in a given $G$-invariant Hermitian inner product $\langle ~,~ \rangle_\pi$ on $\pi$. The correlation coefficient of $\pi$ with respect to $H$ and $K$ is then defined as \cite[\S 8, \S 9]{gro91}
\[c(\pi;H,K) = |\langle v_H,v_K \rangle_\pi|^2.\] 
Since $\langle~,~\rangle_\pi$ is unique up to scalar multiplication, and since $v_H$ and $v_K$ are unique up to a complex number of absolute value $1$ (in a chosen invariant inner product), it follows that $c(\pi;H,K)$ is well-defined.

In this paper, we study the correlation coefficient $c(\pi;H,K)$ of an irreducible complex representation $\pi$ of $G={\rm GL}_2(\mathbb F_q)$ with respect to a split torus $H$ and a non-split torus $K$. This was the theme in \cite{aj22} as well, where the embeddings of the two tori were chosen carefully so that the pair $(H,K)$ inside $G$ is unique up to conjugacy. Consequently, the value of the correlation coefficient $c(\pi;H,K)$ is independent of such compatible pairs $(H,K)$. In this situation, the vanishing of the correlation coefficient is closely related to the root number of $\pi$. When $\pi$ admits fixed vectors for both $H$ and $K$, its root number is $\pm 1$ and it is necessary that the root number is $1$ for $c(\pi; H,K) \neq 0$. When $q=p$, this condition is also sufficient \cite[Corollary 1.2]{aj22}. Thus, roughly half the time the correlation vanishes.  

The present work is inspired by \cite[\S 7.1.1]{bhko24} where the local newform is shown to be a test vector for a local unramified toric period for irreducible supercuspidal representations of ${\rm PGL}_2(\mathbb Q_p)$, where the novelty is that the embedding of the unramified torus in ${\rm PGL}_2(\mathbb Q_p)$ depends only on the conductor of the representation. This is done for conductor (which is necessarily even) $\geq 4$. Moreover, it is shown that the toric period is an $\ell$-adic unit for a given prime $\ell \neq p$. In the case of conductor two, the problem may be reformulated to that of simultaneous nonvanishing of the correlation coefficient $c(\pi;H,K)$ for all irreducible cuspidal representations $\pi$ of ${\rm GL}_2(\mathbb F_p)$ for a choice of $K$ with $H$ being the standard diagonal embedding of the split torus. Motivated by these considerations, it is natural to ask the following question. 
\begin{question}\label{question}
Fix $H$ to be the standard diagonal embedding of the split torus in ${\rm PGL}_2(\mathbb F_q)$. Is there an embedding $K$ of the non-split torus such that the correlation coefficient $c(\pi; H, K)$ is nonvanishing for all irreducible representations $\pi$ of ${\rm PGL}_2(\mathbb F_q)$? And such that this is a $p$-adic unit?
\end{question}
 
With this in mind, we compute the correlation coefficient of $\pi$ with respect to $H$ and a suitable conjugate of $K$ considered in \cite{aj22} by a unipotent element as in [BHKO24]. Thus, let $\alpha \in \mathbb F_q^\times \setminus \mathbb F_q^{\times 2}$ and let 
\[K= K_\alpha = \left\{ \left(\begin{array}{cc} 1 & \alpha z \\ z & 1 \end{array} \right), \left(\begin{array}{cc} 0 & \alpha \\ 1 & 0 \end{array} \right)  \mid z \in \mathbb F_q \right\}.\] Let $u \in \mathbb F_q$ and let
\begin{align*}
K_{\alpha,u} &= \left(\begin{array}{cc} 1 & u \\ 0 & 1 \end{array} \right) K_\alpha \left(\begin{array}{cc} 1 & -u \\ 0 & 1 \end{array} \right). 
\end{align*}
Unlike in \cite{aj22}, where the pair $(H,K_\alpha)$ was carefully chosen in order for $c(\pi;H,K_\alpha)$ to be independent of $\alpha$, and in this case roughly half the time the correlation was zero, now we should expect the answer to depend on both $\alpha$ and $u$. 

We fix an identification of $\mathbb C$ with $\overline{\mathbb Q}_p$ and consider the representations as defined over $\overline{\mathbb Q}_p$. Let $\mathcal O_m \subset \overline{\mathbb Z}_p$ denote the ring of integers of the unramified extension of $\mathbb Q_p$ of degree $m$. Let $\chi$ be a character of $\mathbb F_q^\times$ that reduces mod $p$ to the identity character of $\mathbb F_q^\times$. Similarly, let $\psi$ be a character of $\mathbb F_{q^2}^\times$ that reduces mod $p$ to the identity character of $\mathbb F_{q^2}^\times$. The unique quadratic character of ${\rm PGL}_2(\mathbb F_q)$ is given by $\eta = \chi^{\frac{q-1}{2}} \circ \det$. An irreducible principal series representation, up to isomorphism, of ${\rm PGL}_2(\mathbb F_q)$ is of the form ${\rm Ps}(\chi^r,\chi^{-r})$, where $1 \leq r \leq \frac{q-3}{2}$. Irreducible cuspidal representations of ${\rm PGL}_2(\mathbb F_q)$, up to isomorphism, are parametrized by $\psi^{(q-1)(r+1)}$, where $0 \leq r \leq \frac{q-3}{2}$. These $(q-2)$  representations and the twisted Steinberg representation St$\otimes \eta$ (and of course the trivial representation) are precisely the representations that admit both $H$ and $K$-invariant vectors. Moreover, for these representations, we have multiplicity one for both the invariant spaces. When $\pi = {\rm St}\otimes \eta$, we take $r = (q-1)/2$. 

\begin{theorem}\label{thm1}
Let $\pi = \pi_r$ be an irreducible representation of ${\rm PGL}_2(\mathbb F_q)$ over $\overline{\mathbb Q}_p$ admitting fixed vectors for both $H$ and $K$.  Then, we have $c(\pi_r;H,K_{\alpha,u}) \in \mathcal O_m$ and 
\[c(\pi_r;H,K_{\alpha,u}) \equiv [P_r(u/\sqrt{\alpha})]^2 \mod p, \] 
where $P_r(X) \in \mathbb Z[\frac{1}{2}][X]$ is the $r$-th Legendre polynomial. 
\end{theorem}

Thus, an affirmative answer to Question \ref{question} reduces to the following expectation.
\begin{quote}
There exists $c \in \mathbb F_q^\times \setminus \mathbb F_q^{\times 2}$ such that 
$(x^2-c)$ is not a factor of any of $P_r(x)$ mod $p$ for $0 \leq r \leq (q-1)/2$.
\end{quote}

We verify the above expectation when $m \geq 2$ and thus obtain the following theorem as a corollary to Theorem \ref{thm1} (see \S \ref{simul}).
\begin{theorem}\label{thm2}
Let $q = p^m$ with $m \geq 2$. There exists $u \in \mathbb F_q^\times$ such that $c(\pi;H,K_{\alpha,u}) \in \mathcal O_m^\times$ for all the irreducible representations $\pi$ of ${\rm PGL}_2(\mathbb F_q)$ over $\overline{\mathbb Q}_p$ admitting fixed vectors for both $H$ and $K$. 
\end{theorem}
 
For the proof of Theorem \ref{thm1}, we closely follow \cite[\S 5]{aj22} (see also \cite[Theorem 7.2]{vat23}). The key non-trivial ingredient in {\cite{aj22} is the mod $p$ reduction of irreducible representations of ${\rm GL}_2(\mathbb F_q)$ defined over $\overline{\mathbb Q}_p$ \cite{dia07} which is the approach as well in \cite[Theorem 7.2]{vat23}. 

Legendre polynomials classically appear in representation theory via the following result due to Vilenkin and Dieudonn\'e \cite[Proposition 1.1]{bry92}: in the representation 
${\rm Sym}^{2r}\mathbb R^2 \otimes \det{^{-r}}$ of ${\rm PGL}_2(\mathbb R)$, the matrix coefficient corresponding to 
$g = [a,  b; c, d]$ and $v = x^r y^r$ is given by $P_r(ad+bc/ad-bc)$. Clubbing this result with Theorem \ref{thm1}, we can derive new identities for Legendre polynomials (see Equations  (\ref{vd}) and  (\ref{fe})).  

We end the introduction by noting that the problem of factorising Legendre polynomials modulo $p$ is intimately related to number theory via Hasse invariants of elliptic curves and in turn to class numbers of imaginary quadratic number fields. We refer to \cite{bm04} and the references therein for further details. In \S \ref{legendre-p}, we discuss this theme in some detail and derive two identities in Lemma \ref{lem-legendre}.   

\section{Recap of the earlier work}\label{recap}

In this section we summarize the strategy employed in \cite{aj22}. Let $q=p^m$ and let $r \in \mathbb N$ be such that $0 \leq r \leq q$. Write
\[r = r_0+r_1p+\dots+r_{m-1}p^{m-1}\]
with $0 \leq r_i \leq p-1$ for $0 \leq i \leq m-1$. Let ${\rm Sym}^{r_i} \overline{\mathbb F}_p^2$ denote the $r_i$-th symmetric power representation of the standard representation of ${\rm GL}_2(\mathbb F_q)$. Then 
 \[\rho_r = {\rm Sym}^{r_0} \overline{\mathbb F}_p^2 \otimes {\rm Sym}^{r_1} \overline{\mathbb F}_p^2 \circ {\rm Frob} \otimes \dots \otimes  {\rm Sym}^{r_{m-1}} \overline{\mathbb F}_p^2 \circ {\rm Frob}^{m-1}\] 
is an irreducible $\overline{\mathbb F}_p$-representation of ${\rm GL}_2(\mathbb F_q)$, where Frob is the Frobenius morphism. Any irreducible mod $p$ representation of ${\rm GL}_2(\mathbb F_q)$ is of this form.  

The Jordan-H\"older constituents of the reduction mod $p$ of irreducible representations (over $\overline{\mathbb Q}_p$) of ${\rm GL}_2(\mathbb F_q)$ are computed in \cite[Proposition 1.1 ~\&~ Proposition 1.3]{dia07}. As mentioned in  the introduction, the $\overline{\mathbb  Q}_p$-representations of interest to us are of the form ${\rm Ps}(\chi^r,\chi^{-r})$ or $\pi(\psi^{(q-1)(r+1)})$ or ${\rm St}\otimes \eta$. The Steinberg representation remains irreducible on reduction mod  $p$, whereas in the principal series and cuspidal cases, generically there are $2^m$ Jordan-H\"older factors in their mod $p$ reduction and these are indexed by subsets of $Z/m\mathbb Z$. Among these factors there is precisely one which admits fixed vectors for $H$ and $K$ and this was isolated in \cite[Proposition 4.2 ~\&~ Proposition 4.4]{aj22} which we now state.

\begin{proposition}\label{prop-future}
Let $\pi_r = {\rm Ps}(\chi^r,\chi^{-r})$ or $\pi(\psi^{(q-1)(r+1)})$. Write $r=r_0+r_1p+\dots+r_{m-1}p^{m-1}$ with $0 \leq r_i \leq p-1$ for $0 \leq i \leq m-1$. Let $a_0+a_1p+\dots+a_{m-1}p^{m-1}$ be the $p$-adic expansion of $2r$ (resp. $2r+1$) when $\pi_r = {\rm Ps}(\chi^r,\chi^{-r})$ (resp. $\pi_r = \pi(\psi^{(q-1)(r+1)})$). Let $J \subset \mathbb Z/m\mathbb Z$ be defined as
\[J = \left\{i \mid a_i = 2r_i \mbox{~or~} 2r_i+1 \right\}.\]
Then the component in the mod $p$ reduction of $\pi$ corresponding to $J$ is
\[ \left( \bigotimes_{i \in J} \left( {\rm Sym}^{2r_i} \overline{\mathbb F}_p^2 \otimes \det{^{-r_i}} \overline{\mathbb F}_p^2 \right) \circ {\rm Frob}^i \right) \bigotimes \left(  \bigotimes_{i \notin J} \left( {\rm Sym}^{2p-2-2r_i} \overline{\mathbb F}_p^2 \otimes \det{^{r_i}} \overline{\mathbb F}_p^2 \right) \circ {\rm Frob}^i \right).\]
\end{proposition}

The analysis carried out in \cite{aj22} has an inherent symmetry between $r$ and $p-1-r$ and therefore it suffices to 
look at representations of the form
\[ \rho_r = \bigotimes_{i=0}^{m-1} \left( {\rm Sym}^{2r_i} \overline{\mathbb F}_p^2 \otimes \det{^{-r_i}} \overline{\mathbb F}_p^2 \right) \circ {\rm Frob}^i. \] Suppose $q=p$ and let 
\[\rho = {\rm Sym}^{2r} \overline{\mathbb F}_p^2 \otimes \det{^{-r}} \overline{\mathbb F}_p^2,\]
where $0 \leq r \leq \frac{p-1}{2}$. The action of $\rho$ on the basis $\{x^{2r-i}y^i \mid i = 0, \dots, 2r \}$ is given by
\[\rho \left( \left(\begin{array}{cc} a & b \\ c & d \end{array} \right)\right)x^{2r-i}y^i = (ax+cy)^{2r-i}(bx+dy)^i (ad-bc)^{-r}.\]
We recall the notations in \cite{aj22}. Let
\[H = \left\{\left(\begin{array}{cc} a & 0 \\ 0 & 1 \end{array} \right) \mid a \in \mathbb F_p^\times\right\}.\] Let $\alpha \in \mathbb F_p^\times \setminus \mathbb F_p^{\times 2}$ and define \[K = K_\alpha = \left\{ \left(\begin{array}{cc} 1 & \alpha z \\ z & 1 \end{array} \right), \left(\begin{array}{cc} 0 & \alpha \\ 1 & 0 \end{array} \right)  \mid z \in \mathbb F_p  \right\}.\] 
The suppression of the subscript $\alpha$ in $K_\alpha$ is eventually justified in the context of Proposition \ref{prop-main}. 

The space of fixed vectors in $\rho$ is one dimensional for both $H$ and $K$, and explicitly we can take (see \cite[\S 5.1]{aj22})
\begin{equation*}\label{eqn-vh}
v_H = x^ry^r,
\end{equation*}
and
\begin{equation*}\label{eqn-vk}
v_K = (\alpha x^2-y^2)^r.
\end{equation*}
Letting 
\[X = \frac{1}{|H|}\sum_{h \in H} \rho(h), Y = \frac{1}{|K|}\sum_{k \in K}\rho(k),\]
and for $s$ and $t$ with
\[Xv_K = s v_H \mbox{~and~} Y v_H = t v_K,\]
it is checked in \cite[\S 5.1]{aj22} that 
\begin{equation}\label{eq1}
s =
\begin{cases}
(-1)^{r/2} {r \choose r/2}\alpha^{r/2} &\text{if $r$ is even,} \\
0 &\text{if $r$ is odd,}
\end{cases}
\end{equation}
and
\begin{equation}\label{eq2}
t =
\begin{cases}
-(-1)^{(p-1-r)/2} {p-1-r \choose (p-1-r)/2}\alpha^{(p-1-r)/2} &\text{if $r$ is even,} \\
0 &\text{if $r$ is odd.}
\end{cases}
\end{equation}

The key strategy in \cite{aj22} is to interpret the correlation coefficient in terms of a double sum of the character of the representation \cite[Lemma 2.1]{aj22}:
\begin{lemma}\label{lemma-elementary}
 For a representation $\pi$ of $G$ with one dimensional space of invariant vectors for $H$ and $K$, we have the identity:
 \[|\langle v_H, v_K \rangle_\pi|^2 = \frac{1}{|H|}\frac{1}{|K|}  \sum_{h \in H} \sum_{k \in K} \chi_\pi(hk),\]
 where $\chi_\pi$ denotes the character of $\pi$.
 \end{lemma}
This observation together with its analogue for Brauer characters \cite[Lemma 5.1]{aj22} can be used to conclude that $c(\pi;H,K) \mod p $ is just $st$ \cite[Proposition 5.2]{aj22}. Thus, it is proved  that (see \cite[Proposition 5.3]{aj22}):
\begin{proposition}
Let $\pi$ be an irreducible representation of ${\rm PGL}_2(\mathbb F_p)$ over $\overline{\mathbb Q}_p$ admitting fixed vectors for both $H$ and $K$.  Thus, $\pi={\rm Ps}(\chi^r,\chi^{-r})$ or $\pi=\pi(\psi^{(p-1)(r+1)})$ or $\pi = {\rm St}\otimes \eta$ in which case we take $r = (p-1)/2$. Then the correlation coefficient $c(\pi;H,K)$ belongs to ${\mathbb Z}_p$ and moreover
\[c(\pi;H,K) \mod p = \begin{cases}
(-1)^{\frac{p-1}{2}} {r \choose \frac{r}{2}} {p-1-r \choose \frac{p-1-r}{2}} &\text{if $r$ is even,} \\
0 &\text{if $r$ is odd.}
\end{cases}
\]
\end{proposition}
\noindent
By appealing to Lucas’s theorem on binomial coefficients modulo $p$, the above arguments can be extended and we get (see \cite[Theorem 5.4]{aj22}):
\begin{proposition}\label{prop-main}
Let $\pi$ be an irreducible representation of ${\rm PGL}_2(\mathbb F_q)$ over $\overline{\mathbb Q}_p$ admitting fixed vectors for both $H$ and $K$.  Thus, $\pi={\rm Ps}(\chi^r,\chi^{-r})$ or $\pi=\pi(\psi^{(q-1)(r+1)})$ or $\pi = {\rm St}\otimes \eta$ in which case we take $r = (q-1)/2$. Write
$r= r_0+r_1p+\dots+r_{m-1}p^{m-1}$. Then the correlation coefficient $c(\pi;H,K)$ belongs to ${\mathbb Z}_p$ and moreover
\[c(\pi;H,K) \mod p =  \begin{cases}
(-1)^{\frac{q-1}{2}} {r \choose \frac{r}{2}} {q-1-r \choose \frac{q-1-r}{2}} &\text{if each $r_i$ is even,} \\
0 &\text{otherwise.}
\end{cases}
\]
\end{proposition}

\section{Legendre Polynomials}\label{s-legendre}

As is well-known, these polynomials can be defined in various ways. For instance, via a generating function  
\[\frac{1}{\sqrt{1-2xt+t^2}} = \sum_{r=0}^\infty P_r(x)t^r,\]  
or via a recursion formula (with $P_0(x) = 1, P_1(x) =x$)   
\[(r+1)P_{r+1}(x) = (2r+1)xP_r(x)-rP_{r-1}(x),\] 
or as a solution of the differential equation   
\[(1-x^2)y^{\prime\prime} - 2xy^\prime+r(r+1)y=0,\] or by applying the Gram-Schmidt process to the standard basis for $\mathbb R[x]$ with respect to $\langle f, g \rangle = \int_{-1}^1 f(x)g(x)dx$.
There are many explicit expressions for these polynomials out of which we write down two, which will be of use to us in this paper.   
Let \[S_k = 
\begin{cases}
\{0, 2, \dots, k\} &\text{if $k$ is even,} \\
\{1,3, \dots, k\} &\text{if $k$ is odd.}
\end{cases}
\]  
We have
\begin{equation}\label{legendre-1}
P_r(x) = \frac{1}{2^r}\displaystyle{\sum_{i \in S_r}} (-1)^{\frac{r-i}{2}} {r \choose (r-i)/2}{r+i \choose i} x^i,
\end{equation}
and 
\begin{equation}\label{legendre-2}
P_r(x) = \frac{1}{2^r}
\displaystyle{\sum_{i \in S_r}} (-1)^{\frac{r-i}{2}} {r \choose i}{r-i \choose (r-i)/2}(1-x^2)^{(r-i)/2} (2x)^i. 
\end{equation}
They satisfy the following two congruences (see \cite[p. 81]{bm04}, \cite[Corollary 1.3]{bry92}). The Ille-Schur congruence states that if $r=\sum_{i=0}^{m-1} r_ip^i$, then   
\begin{equation}\label{ille-schur}
P_r(x) \equiv \prod_{i=0}^{m-1} P_{r_i}(x)^{p^i} \mod p.
\end{equation}
There is also the symmetry around $(p-1)/2$ given by   
\begin{equation}\label{symmetry}
P_{p-1-r}(x) \equiv P_r(x) \mod p.
\end{equation}
 
\section{Proof of Theorem \ref{thm1}}

Now we prove Theorem \ref{thm1} following the strategy in \cite{aj22} recalled in \S \ref{recap}. We first handle the case of one tensor factor, that is to say $r_i = 0$ for $i \geq 1$ in
\[ \rho_r = \bigotimes_{i=0}^{m-1} \left( {\rm Sym}^{2r_i} \overline{\mathbb F}_p^2 \otimes \det{^{-r_i}} \overline{\mathbb F}_p^2 \right) \circ {\rm Frob}^i\] 
 and then take up the general case. In the general case, the Ille-Schur congruence for Legendre polynomials (cf.  (\ref{ille-schur})) will play the role of Lucas's theorem in \cite[\S 5.2]{aj22} (see Section \ref{recap}). It is convenient to start with the case $q=p$ for clarity and simplicity. 

\subsection{The case $q=p$}\label{sub1}

Let $u \in \mathbb F_p$ and let
\begin{align*}
K_{\alpha,u} &= \left(\begin{array}{cc} 1 & u \\ 0 & 1 \end{array} \right) K_\alpha \left(\begin{array}{cc} 1 & -u \\ 0 & 1 \end{array} \right) \\
&=
\left\{ \left(\begin{array}{cc} 1+uz & \alpha z - u^2 z \\ z & 1-uz \end{array} \right), \left(\begin{array}{cc} u & \alpha - u^2 \\ 1 & -u \end{array} \right)  \mid z \in \mathbb F_p  \right\}.
\end{align*}
Following the approach of \cite[\S 5.1]{aj22}, we compute $s_r(u,\alpha)$ in  
\[X \cdot \pi\left(\left(\begin{array}{cc} 1 & u \\ 0 & 1 \end{array} \right) \right)v_{K_\alpha}  = s_r(u,\alpha) \cdot x^ry^r \]
and the action on $v_H$ of 
\[Y_u = \pi\left( \left(\begin{array}{cc} 1 & u \\ 0 & 1 \end{array} \right) \right) \cdot Y \cdot \pi \left( \left(\begin{array}{cc} 1 & -u \\ 0 & 1 \end{array} \right) \right),\]
which should be a multiple, say $t_r(u,\alpha)$, of 
\[v_{K_{\alpha,u}} = \pi\left( \left(\begin{array}{cc} 1 & u \\ 0 & 1 \end{array} \right) \right) v_K = (\alpha x^2 - (ux+y)^2)^r.\]

First we compute $t_r(u,\alpha)$. Note that
\begin{align*}
Y_u \cdot x^ry^r &= \left[ \left(\begin{array}{cc} u & \alpha - u^2 \\ 1 & -u \end{array}\right) \cdot x^ry^r + \sum_{z \in \mathbb F_p} \left(\begin{array}{cc} 1+uz & \alpha z - u^2 z \\ z & 1-uz \end{array}\right) \cdot x^ry^r   \right] 
\end{align*}
\begin{align*}
=  \left[ (-\alpha)^{p-1-r} (ux+y)^r ((\alpha-u^2)x-uy)^r + \sum_{z \in \mathbb F_p} \frac{((1+uz)x+zy)^r((\alpha z-u^2 z)x+(1-uz)y)^r}{(1-\alpha z^2)^r} \right].
\end{align*}
Equating this expression to 
\[t_r(u,\alpha) \cdot (\alpha x^2 - (ux+y)^2)^r,\]
and looking at the coefficient of $x^{2r}$, we get
\begin{align*}
t_r(u,\alpha) &= (-\alpha)^{p-1-r}u^r+\sum_{z\in \mathbb F_p} \frac{z^r(1+uz)^r}{(1-\alpha z^2)^r} \\
&=  (-\alpha)^{p-1-r}u^r+ \sum_{z \in \mathbb F_p^\times} (z^{-1}-\alpha z)^{p-1-r}(1+uz)^r  \\
& = - \displaystyle{\sum_{j \in S_{p-1-r}}} (-1)^{(p-1-r-j)/2}{p-1-r \choose (p-1-r-j)/2} {r \choose j}\alpha^{(p-1-r-j)/2} u^j
\end{align*}
where, as in \S \ref{s-legendre},
\[S_k = 
\begin{cases}
\{0, 2, \dots, k\} &\text{if $k$ is even,} \\
\{1,3, \dots, k\} &\text{if $k$ is odd.}
\end{cases}
\]
Note that when $u = 0$, we have $t_r(u,\alpha) = t$ (cf. (\ref{eq2})). Noting that 
\[{r \choose j} \equiv (-1)^j {p-1-r+j \choose j} \mod p,\] it follows from (\ref{legendre-1}), as $j$ and $p-1-r$ are of the same parity, that
\begin{equation}\label{eq-tru}
t_r(u,\alpha) \equiv -(-1)^{p-1-r}2^{p-1-r}\alpha^{(p-1-r)/2}P_{p-1-r}(u/\sqrt{\alpha}) \mod p.
\end{equation}

Next we compute $s_r(u,\alpha)$. We have
\[X \cdot v_{K_{\alpha,u}} = s_r(u,\alpha) \cdot v_H.\]
Note that
\[
\left(\begin{array}{cc} a & 0 \\ 0 & 1 \end{array} \right) \cdot (\alpha x^2-y^2)^r = (\alpha a x^2-a^{-1}y^2)^r,\]
and therefore
\begin{align*}
X \left(\begin{array}{cc} 1 & u \\ 0 & 1 \end{array} \right) v_K &= -\displaystyle{\sum_{a \in \mathbb F_p^\times}} \left(\begin{array}{cc} 1 & a u \\ 0 & 1 \end{array} \right)  (\alpha a x^2-a^{-1}y^2)^r \\
&= -\displaystyle{\sum_{a \in \mathbb F_p^\times}} (\alpha a x^2-a^{-1}(aux+y)^2)^r \\
&= - \displaystyle{\sum_{a \in \mathbb F_p^\times}} ((\alpha - u^2) a x^2 - 2 u xy - a^{-1}y^2)^r \\ 
&= s_r(u,\alpha) \cdot x^ry^r,
\end{align*}
where, noting that $i$ and $r$ are of the same parity, 
\[s_r(u,\alpha) = (-1)^r
\displaystyle{\sum_{i \in S_r}} (-1)^{(r-i)/2} {r \choose i}{r-i \choose (r-i)/2}(\alpha - u^2)^{(r-i)/2} (2u)^i. 
\] In the above computation we started with a negative sign to accommodate $|H| = p-1$ in the definition of $X$. 
Note that when $u = 0$, we have $s_r(u,\alpha) = s$ (cf. (\ref{eq1})). It follows from (\ref{legendre-2}) that
\begin{equation}\label{eq-sru}
s_r(u,\alpha) \equiv (-1)^r 2^r \alpha^{r/2} P_r(u/\sqrt{\alpha}) \mod p.
\end{equation}

As outlined in \S \ref{recap}, the correlation coefficient is the product $s_r(u,\alpha)t_r(u,\alpha)$. Thus, combining the identities (\ref{eq-tru}) and (\ref{eq-sru}), and making use of (\ref{symmetry}), we have proved:
\begin{proposition}\label{prop-2}
Let $\pi$ be an irreducible representation of ${\rm PGL}_2(\mathbb F_p)$ over $\overline{\mathbb Q}_p$ admitting fixed vectors for both $H$ and $K_{\alpha,u}$.  Thus, $\pi={\rm Ps}(\chi^r,\chi^{-r})$ or $\pi=\pi(\psi^{(p-1)(r+1)})$ or $\pi = {\rm St}\otimes \eta$ in which case we take $r = (p-1)/2$. Then the correlation coefficient $c(\pi;H,K_{\alpha,u})$ belongs to ${\mathbb Z}_p$ and \[c(\pi;H,K_{\alpha,u}) \equiv [P_r(u/\sqrt{\alpha})]^2 \mod p.\]
\end{proposition}

\begin{remark}\label{rmk-symmetry}
The symmetry around $(p-1)/2$ stated as (\ref{symmetry}) can in fact be seen as  follows. If $\pi$ is thought of as a representation of ${\rm PGL}_2(\mathbb F_p)$ over $\mathbb C$ then the complex numbers $s$ and $t$ determined by the corresponding averaging operators on $\pi$ satisfies $\overline{s} = t$ \cite[\S 2]{aj22}. Since complex conjugation, after an identification with $\overline{\mathbb Q}_p$, corresponds to the involution 
$r \mapsto p-1-r$ on the weights in reduction mod $p$ it follows that one of them is obtained by the other.
\end{remark}

\begin{remark}[Archimedean analogue]\label{arch}
The proof of Proposition \ref{prop-2} is purely algebraic and it goes through for finite dimensional representations $\pi_r = {\rm Sym}^{2r} \mathbb R^2  \otimes \det^{-r}$ of ${\rm PGL}_2(\mathbb R)$. Thus we get, for $H =\mathbb R^\times$, $K = \mathbb C^\times$ and $u \in \mathbb R$, $|\langle v_H, v_{K_u} \rangle |^2  = P_r(\imath u)^2$.
\end{remark}

\subsection{The case $r_i = 0$ for $i \geq 1$} 

In this subsection, $q=p^m$ but we assume $r_i =0$ for $i \geq 1$. Now $u \in \mathbb F_q$ and $\alpha \in \mathbb F_q^\times \setminus \mathbb F_q^{\times 2}$. Denoting $r_0$ as $r$, everything goes through in \S \ref{sub1} with 
\begin{equation}\label{eq-tru-1}
t_r(u,\alpha) \equiv -(-1)^{q-1-r}2^{q-1-r}\alpha^{(q-1-r)/2}P_{q-1-r}(u/\sqrt{\alpha}) \mod p.
\end{equation}
instead of (\ref{eq-tru}). The factor $s_r(u,\alpha)$ has the same expression which we state again:
\begin{equation}\label{eq-sru-1}
s_r(u,\alpha) \equiv (-1)^r 2^r \alpha^{r/2} P_r(u/\sqrt{\alpha}) \mod p.
\end{equation} 
Thus, 
\[s_r(u,\alpha)t_r(u,\alpha) = P_r(u/\sqrt{\alpha}) P_{q-1-r}(u/\sqrt{\alpha}). \]
Since $r=r_0$, we have $q-1-r = (p-1-r)+\sum_{i=0}^{m-1} (p-1)p^i$ and therefore by (\ref{ille-schur}) and (\ref{symmetry}) we see that $P_{q-1-r}(x)  \equiv P_r(x) \mod  p$. Thus, once again,
\[s_r(u,\alpha)t_r(u,\alpha) = [P_r(u/\sqrt{\alpha})]^2  \mod p.\]

\subsection{The general case} We follow the strategy in \cite[\S 5.2]{aj22}. As mentioned in \S \ref{recap}, for $q=p^m$, we are led to analyse a representation of ${\rm PGL}_2(\mathbb F_q)$ of the form
\[ \rho_r = \bigotimes_{i=0}^{m-1} \left( {\rm Sym}^{2r_i} \overline{\mathbb F}_p^2 \otimes \det{^{-r_i}} \overline{\mathbb F}_p^2 \right) \circ {\rm Frob}^i. \] 
In this case, we have 
\[v_H = \bigotimes_{i=0}^{m-1} x_i^{r_i}y_i^{r_i}\]
and
\[v_{K_{\alpha,u}} = \bigotimes_{i=0}^{m-1} (\alpha^{p^i} x_i^2-(u^{p^i}x_i+y_i)^2)^{r_i}.\]
With $X,Y_u,s_r(u,\alpha),t_r(u,\alpha)$ defined as earlier and by denoting the corresponding scalars for the $i$-th component by $s_i(u,\alpha)$ and $t_i(u,\alpha)$, we get 
\begin{align*}
s_r(u,\alpha)t_r(u,\alpha) &= \prod_{i=0}^{m-1} s_i(u,\alpha)^{p^i}t_i(u,\alpha)^{p^i} \\
&= \prod_{i=0}^{m-1} [P_{r_i}(u/\sqrt{\alpha})]^{2p^i} \\
&= [P_r(u/\sqrt{\alpha})]^2
\end{align*}
where the last step is a consequence of the Ille-Schur congruence (cf. (\ref{ille-schur})). This proves Theorem \ref{thm1}.

\section{Simultaneous nonvanishing}\label{simul}

In this section, we prove Theorem \ref{thm2}  in which we have assumed that $m \geq 2$. In order to get simultaneous nonvanishing of the correlation coefficient for all the irreducible representations with fixed vectors for $H$ and $K$, we need to exhibit $u \in \mathbb F_q$ such that 
\[P_r(u/\sqrt{\alpha}) \not\equiv 0 \mod p\]
for all $0 \leq r \leq (q-1)/2$. By the Ille-Schur and symmetry congruences (\ref{ille-schur})  and (\ref{symmetry}) for Legendre polynomials, $P_r(x) \mod p$, for any $r$, is completely determined by $P_i(x) \mod p$ for $0 \leq i \leq (p-1)/2$. Thus, to prove Theorem \ref{thm2}, we only need to note that the product of these $(p-1)/2$ polynomials, say,
\[P(x) = \prod_{i=0}^{(p-1)/2} P_i(x)\] 
does not have a factor of the form $(x^2-c)$ for some $c \in \mathbb F_q^\times \setminus \mathbb F_q^{\times 2}$. Indeed, if this is the case then we can take $u = \sqrt{\alpha c} \in \mathbb F_q$. As these polynomials have terms only of even degree or only of odd degree, considering $P(x)/x$ if necessary, we can treat $P(x)$ as a polynomial in $x^2$. Note that the degree of $P(t=x^2)$ is $\lfloor(p^2-1)/16\rfloor$ and $\mathbb F_q$ has $(p^m-1)/2$ quadratic non-residues and thus we are done if $m \geq 2$.

\section{Legendre Polynomials mod $p$}\label{legendre-p}

As mentioned in \S \ref{intro}, the reduction modulo $p$ of Legendre polynomials is connected to number theory \cite{bm04}. More specifically, the reduction mod $p$ of the Legendre  polynomial in degree $(p-1)/2$,  the one corresponding to the correlation coefficient for  the Steinberg representation  of ${\rm  PGL}_2(\mathbb F_p)$, is related  to the Hasse invariant  for  elliptic curves.  It is  known  that this particular Legendre polynomial has only linear and  quadratic  factors in its reduction modulo $p$ and  there are precise formulas of the number of factors of each type which involve  the class number of $\mathbb  Q(\sqrt{-p})$ \cite[Theorem 1]{bm04}.

In this  section, we make a  few observations  about Legendre polynomials which are consequences of Theorem \ref{thm1}. 

For the first set of observations, we  combine  Theorem \ref{thm1} with the result due to Vilenkin and Dieudonn\'e mentioned in the introduction. Applying this result  to each element in $K_{\alpha,u}$ and taking the sum, on the one side we get the correlation coefficient and on the other side a sum of $(p+1)$ many explicit matrix coefficients. This leads to the identity
\begin{equation}\label{vd}
P_r(x)^2 \equiv 1+P_r(2x^2-1)+2\sum_{y \in \mathbb F_p^\times \setminus \mathbb F_p^{\times  2}} P_r\left(\frac{y+1-2x^2}{y-1} \right) \mod p,
\end{equation}
where the  sum ranges over quadratic non-residues in $\mathbb F_p^\times$.
Expanding $1/(y-1)$ as $(y-1)^{p-2}$ and taking the sum over quadratic non-residues we get the identity
\[P_r(x)^2 = \sum_{i \in S_r} a_i Q_i(2x^2-1)\]
where $P_r(x) = \sum_{i \in S_r}a_ix^i$ and
\[Q_i(x) = \sum_{j=0}^i (-1)^j {i \choose j} {i - j -\frac{1}{2} \choose i} x^j.\] This is initially obtained modulo $p$ but it lifts to characteristic zero. Since, $Q_i(x) = x^i P_i(\frac{x+x^{-1}}{2})$ (see \cite[p. 1]{kel59}), we get:
\begin{equation}\label{fe}
P_r(x)^2 = \sum_{i \in S_r} a_i (\sqrt{2x^2-1})^i P_i\left(\frac{x^2}{\sqrt{2x^2-1}} \right).
\end{equation}

The approach of \cite{bm04} gives  information on reduction modulo $p$ also of $P_r(x)$ for $r=(p-e)/k$ for $k = 3,4$ where $p \equiv e \mod k$. This is because of certain congruences between these Legendre polynomials and the one for $(p-1)/2$ \cite[p. 81]{bm04}. In degrees other than these, there does  not seem to be any pattern in the  reduction modulo  $p$ of $P_r(x)$. 

As a consequence of the present work, it is easy to derive the following two identities.  Probably these are well-known. 
\begin{lemma}\label{lem-legendre}
Modulo p, we have
\begin{enumerate}
\item $\displaystyle{\sum_{r=0}^{p-1}} P_r(x)  = \left( \frac{-2}{p}\right)(x-1)^{(p-1)/2}$,
\item $\displaystyle{\sum_{r=0}^{p-1}} P_r(x)^2  =  (x^2-1)^{(p-1)/2}$.
\end{enumerate} 
\end{lemma}  
The main ingredient in proving the second identity is the character  relation stated in Lemma \ref{lemma-elementary}. For the first identity,  we need an appropriately modified version of \cite[Proposition 7.1]{aj22} which  we recall below.
\begin{proposition}\label{prop-ip}  
Let $v_H$ and $v_K$ denote respectively the $H$ and  $K$ fixed vectors (normalized)  in  the principal series representation $\pi=  {\rm Ps}(\chi,\chi^{-1})$ of ${\rm PGL}_2(\mathbb F_q)$. Let $\alpha \in \mathbb F_q \setminus \mathbb F_q^{\times 2}$. Then, we have
\[\langle v_H, \left(\begin{array}{cc} 1 & u \\  0 &  1 \end{array} \right)  \cdot v_K \rangle_\pi = - \sum_{\lambda \in \mathbb F_q^\times}\chi\left(\frac{\alpha -(\lambda+u)^2}{\lambda} \right).\]
 \end{proposition}
\begin{proof}[Proof of Lemma \ref{lem-legendre}]
Modulo $p$, the left  hand side of the identity in Proposition \ref{prop-ip}  is (cf. (\ref{eq-sru}))
\[s_r(u, \alpha) = (-1)^r2^r\alpha^{r/2}P_r(u/\sqrt{\alpha}),\]
if $\chi$ mod $p$ is the $r$-th power of the identity  character of  $\mathbb F_p^\times$. Now,
\begin{align*}
-\sum_{\chi \in \widehat{\mathbb F_p^\times}} \sum_{\lambda \in \mathbb F_p^\times}\chi\left(\frac{\alpha -(\lambda+u)^2}{\lambda} \right) &= -\sum_{\lambda \in \mathbb F_p^\times} \left(\sum_{\chi \in \widehat{\mathbb F_p^\times}}  \chi\left(\frac{\alpha -(\lambda+u)^2}{\lambda} \right)\right) \\
&=  \begin{cases}
0 &\text{$4u+4\alpha+1 \notin \mathbb F_p^{\times 2}$,} \\
-(p-1) &\text{$4u+4\alpha+1 =0$,} \\
-2(p-1)  &\text{$4u+4\alpha+1 \in  \mathbb F_p^{\times 2}$.} 
\end{cases}
\end{align*}
Thus,
\[\sum_{r=0}^{p-1} s_r(u,\alpha)  = -1+  \sum_{r=0}^{p-2} s_r(u,\alpha) = (4u+4\alpha+1)^{\frac{p-1}{2}}.   \]
We may treat this as a formal identity and plug in $\alpha =  1/4$ to get
\[\sum_{r=0}^{p-1} P_r(-2u) = (4u+2)^{\frac{p-1}{2}}. \] Put $x=-2u$ and the first identity follows.

For the proof of the second  identity, we make use of the Gelfand-Graev representation of ${\rm PGL}_2(\mathbb  F_q)$. It is well-known that this representation has only two non-zero character values and these are given by $(q^2-1)$ at the identity and $-1$ at the (unique) unipotent conjugacy class. It can be checked that the cardinality of 
\[\left\{(h,k_\alpha) \mid h\left(\begin{array}{cc} 1 & u \\ 0 & 1 \end{array}\right)k_\alpha \left(\begin{array}{cc} 1 & -u \\ 0 & 1 \end{array}\right) \mbox{~is unipotent} \right\} \]
is given by (cf. \cite[Remark 17]{aj22})
\[ \begin{cases}
q-1 &\text{if $p \equiv 1 \mod 4$ ~\&~ $\alpha - u^2 \notin \mathbb F_q^{\times 2}$,} \\
q-3 &\text{if $p \equiv 3 \mod 4$ ~\&~ $\alpha - u^2 \notin \mathbb F_q^{\times 2}$,} \\
q-3 &\text{if $p \equiv 1 \mod 4$ ~\&~ $\alpha - u^2 \in \mathbb F_q^{\times 2}$,}  \\
q-1 &\text{if $p \equiv 3 \mod 4$ ~\&~ $\alpha - u^2 \in \mathbb F_q^{\times 2}$.} 
\end{cases}\]  
Thus, the character double sum on the right hand side of the identity in Lemma \ref{lemma-elementary} for the Gelfand-Graev representation gives the value
\[
\begin{cases}
q/(q+1) &\text{in Cases I and IV,} \\
(q^2-q+2)/(q^2-1) &\text{in Cases II and III. }
\end{cases}
\]
Since the Gelfand-Graev representation contains every irreducible representation other than $1$ and $\eta$ with multiplicity one, and noting that the Steinberg representation does not have a $K$-fixed vector, it follows again from Lemma \ref{lemma-elementary} that 
\[\sum_{\pi} | \langle v_H, v_{K_{\alpha,u}} \rangle |_\pi^2 = 
\begin{cases}
0 &\text{in Cases I and IV,} \\
-2 &\text{in Cases II and III,}
\end{cases}
\] 
modulo  $p$. By Theorem \ref{thm1}, the left hand side is 
\[P_{\frac{p-1}{2}}(u/\sqrt{\alpha})^2+2\sum_{r=1}^{(p-3)/2}  P_r(u/\sqrt{\alpha})^2\]
and therefore it follows from (\ref{symmetry}) that
\[\sum_{r=0}^{p-1} P_r(u/\sqrt{\alpha})^2 = -(u^2-\alpha)^{(p-1)/2}.\]
Taking $x=u/\sqrt{\alpha}$ and noting that $\alpha^{\frac{p-1}{2}} = -1$, we get the second identity. 
\end{proof}
\begin{remark}
It is a pleasant exercise to sum Equation (\ref{vd}) over $0 \leq r \leq  p-1$ and verify the resulting equality by making use  of the two identities of Lemma \ref{lem-legendre}.
\end{remark}
\begin{remark}
In light of Lemma \ref{lem-legendre}, it is natural to look at "the higher moments" given by $G_i(x) = \sum_{r=0}^{p-1} P_r(x)^i$. It seems to be the case that $G_3(x)$ admits only linear and quadratic factors in its reduction modulo $p$. 
\end{remark}   
\begin{remark}
The expectation stated in the introduction, which we proved in \S \ref{simul} only under the assumption that $q > p$,  is equivalent to the assertion that $G_{2(p-1)}$ admits a  factor of the form  $(x^2-c)$ for some $c \in \mathbb F_q^\times \setminus \mathbb  F_q^{\times 2}$, in its reduction modulo $p$. We are unable to verify this in this paper for $q=p$.
\end{remark}
 
\section*{Acknowledgements} 

The author would like to warmly thank Dipendra Prasad for suggesting this problem and for many helpful conversations and encouragement. Thanks  are  due to Arindam Jana for a careful reading of  the manuscript  and for his comments.


\end{document}